\documentclass{amsart}

\usepackage[T1]{fontenc}

\usepackage{amscd,amsmath,amsthm,amssymb,amsfonts}
\usepackage{graphicx}
\usepackage{enumitem}
\usepackage{paralist}
\usepackage{csquotes}
\usepackage{tikz-cd}
\usepackage{cleveref}
\usepackage{url}

\newtheorem{thm}{Theorem}[subsection]
\newtheorem{conj}[thm]{Conjecture}
\newtheorem{rmk}[thm]{Remark}
\newtheorem{defi}[thm]{Definition}
\newtheorem{lem}[thm]{Lemma}

\renewcommand{\AA}{\mathbb A}
\newcommand{\PP}{\mathbb P}
\newcommand{\CC}{\mathbb C}
\newcommand{\RR}{\mathbb R}
\newcommand{\ZZ}{\mathbb Z}
\newcommand{\RP}{\mathbb{RP}}
\newcommand{\CP}{\mathbb{CP}}
\newcommand{\Pic}{\operatorname{Pic}}
\DeclareMathOperator{\CH}{CH}
\newcommand{\Spec}{\operatorname{Spec}}

\begin{document}

\title{A cellular (co)homology computation for $\overline{M_{0,n}}$}
\author{Jan Hennig}
\address{\parbox{\linewidth}{Heinrich-Heine-Universität Düsseldorf,\\ Universitätsstraße 1, 40225 Düsseldorf, Germany}} 
\email{jan.hennig@uni-duesseldorf.de}        

\begin{abstract}
In this article we set up and showcase cellular computations for (co)homology with values in strictly $\AA^1$-invariant sheaves. These computations encapsulate many classical invariants like Chow groups and singular cohomology of the real points. They also extend enumerative arguments from algebraically closed fields to more general fields. The spaces considered here have to admit a cellular structure. Instead of using the classical notion of cellularity, i.e.\ having a stratification by affine spaces, more general stratifications by cohomologically trivial spaces are used, following Morel--Sawant.

Examples of cellular spaces include projective spaces and their products, but also spaces such as $\overline{M_{0,n}}$, the moduli space of stable genus $0$ curves with $n$ marked points. For these examples, we showcase the computations and show how to derive the classical results. 

Hopefully, the following text provides enough evidence to be convincing that such computations are doable and is encouraging to start computing the cohomology for more cellular spaces. This is part of the author's PhD thesis.

\end{abstract}

\maketitle    

\section{Motivation, classical setting and preliminaries}
\subsection{Why do we care?}
\subsubsection{Cellular decompositions}
In classical topology, open balls are arguably the easiest topological spaces one can consider. Using these \enquote{cells} as building blocks, one can form CW-complexes. A CW-complex can be studied by examining how these cells are glued together. Every manifold has the homotopy-type of a CW-complex, therefore there is not much harm in restricting to CW-complexes. In algebraic geometry, when studying schemes or varieties, this is not so easy. A strictly cellular structure is a way to build a scheme out of affine spaces $\AA^n_k$, the easiest possible cells. This structure allows for similar arguments as for CW-complexes. But there are many schemes that do not admit such a strictly cellular structure. For example, elliptic curves do not contain an $\AA^1_k$ (they are not rational), and hence cannot be strictly cellular\footnote{Elliptic curves $E$ will not be cellular in the more general notion either. Any open one-dimensional subscheme $U\subseteq E$ over $\CC$ has uncountable $\CH^1(U) = H^1(U, \underline{K}^M_1)$, hence is not cohomologically trivial.}. 

For the spaces we would like to study, it will be useful to allow more cells than $\AA^n_k$. These cells will be called \enquote{cohomologically trivial} and their key property is that they have the same cohomological vanishing behavior as affine spaces. For example, cells such as $\mathbb{G}_m$ are now allowed. Given a cellular structure built out of these cells again allows for similar arguments as for CW-complexes. The trade-off is that the cells require more study.

\subsubsection{Milnor-Witt (co)homology and $\overline{M_{0,n}}$}
Computations in Chow rings are used to answer classical enumerative questions like: 
\begin{quote}
    How many rational degree $d$ curves pass through $3d-1$ given points in general position in $\PP^2_\CC$?
\end{quote}
Small examples of this question are easy, giving \enquote{exactly one line through two points} and \enquote{exactly one conic through 5 points}. Higher degrees become more complicated, as there are exactly $12$ rational cubics through $8$ points. The general case can be answered by using the computation of $\CH^*(\overline{M_{0,n}})$, see \cite{KontsevichManin}.

One big downside of these computations is, that they are only valid over algebraically closed fields. 

There are exactly $27$ lines on a smooth cubic surface in $\PP^3_\CC$. Over the reals there could be $3$, $7$, $15$, or $27$ lines (and all those configurations do exist). 

Bloch's formula $\CH^q(X) \cong H^q(X,\underline{K}_q^M)$ describes the Chow ring as the (Nisnevich or Zariski) cohomology of the Milnor K-theory sheaf $\underline{K}_\ast^M$. Using the Milnor-Witt K-theory sheaf $\underline{K}_\ast^{MW}$ instead, introduces more arithmetic information of the base field. For example the value of a degree is not an integer anymore but an element in $\operatorname{GW}(k)$. These turn classical enumerative counts into signed counts valid also for non-algebraically closed fields. 

\subsection{The space $\overline{M_{0,n}}$} 
We start by looking at $M_{0,n}$, the moduli space of smooth genus $0$ curves with $n$ marked points. A smooth connected genus $0$ curve with a rational point is isomorphic to $\PP^1_k$. A (smooth) genus $0$ curve with $n$ marked points is a tuple $(C; p_1,\dots, p_n)$, where $C\cong \PP^1_k$ and $p_1,\dots,p_n\in C$ are distinct points. A morphism of marked curves $f\colon (C; p_1,\dots,p_n)\to (C';q_1,\dots,q_n)$ is a morphism of curves $f\colon C\to C'$ such that $f(p_i)=q_i$ for $i=1,\dots,n$. Note that the order of the markings is important. 

Given three distinct points in $\PP^1_k$, these can be sent via a unique automorphism, a Möbius transformation, to $0$, $1$, and $\infty$. Therefore, there is exactly a single isomorphism class of (smooth) genus $0$ curves with $3$ marked points, namely $(\PP^1_k;0,1,\infty)$. 

If we want to describe $M_{0,n}$, the space of isomorphism classes of (smooth) genus $0$ curves with $n$ marked points, we can send the first three markings to $0$, $1$ and $\infty$, and use the others as coordinates, so
\[M_{0,n} \cong \left\{(x_1,\dots,x_{n-3})\in (\PP^1_k)^{n-3} ~|~ \forall i: x_i\notin\{0,1,\infty\}~\text{ and }~ \forall i\neq j: x_i\neq x_j\right\}.\]
This space is nice because it is a smooth affine variety. It can be viewed as a complement of hyperplanes in $\AA^{n-3}_k$, a fact we will come back to. 

The downside is that it is not proper. A possible compactification one could take is $(\PP^1_k)^{n-3}$. Still, we would like to give geometric meaning to the boundary points, which will correspond to (hopefully, only mildly) singular curves. This is the reason we are taking the Deligne-Mumford compactification $\overline{M_{0,n}}$ instead.

Points in $\overline{M_{0,n}}$ will not only parametrize smooth marked curves, but stable curves. A stable curve $(C; p_1,\dots, p_n)$ of genus $0$ with $n$ marked points is a curve $C$ with closed points $p_1,\dots,p_n\in C$ such that:
	\begin{compactenum}[(i)]
		\item $C$ has finitely many irreducible components $C_i$, which are all isomorphic to $\PP^1_k$,
		\item for $i\neq j$ the intersection $C_i\cap C_j$ is either empty or an ordinary double point,
		\item the graph of $C$, which has the irreducible components as vertices and their intersections as edges, is a tree,
		\item the points $p_1,\dots,p_n\in C$ are distinct,
		\item each marked point $p_i$ lies on exactly one component $C_j$,
		\item each component $C_i$ contains at least 3 special points, i.e.\ markings or intersection points with other components. 
	\end{compactenum}
With $n\geq 3$ and these definitions in place, the space $\overline{M_{0,n}}$ turns out to be a smooth projective variety and a fine moduli space for the corresponding moduli problem, which is not too far from the naive compactification $(\PP^1_k)^{n-3}$. It is an iterated blow-up of $(\PP^1_k)^{n-3}$ along regularly embedded codimension $2$ subschemes, which admit a combinatorial description \cite{Keel}.

The smallest example is $\overline{M_{0,3}} \cong \ast$, as there are no new stable curves compared to $M_{0,3}$. For $M_{0,4} \cong \PP^1_k\setminus\{0,1,\infty\}$ the compactification is $\overline{M_{0,4}}\cong \PP^1_k$. The three added stable curves all have two components with two markings each. The space $\overline{M_{0,5}}$ can be described as the blow-up of $\PP^1_k \times \PP^1_k$ in $\{(0,0),(1,1),(\infty,\infty)\}$.
The boundary $\partial \overline{M_{0,n}} = \overline{M_{0,n}}\setminus M_{0,n}$ admits a good combinatorial description. It is a simple normal crossing divisor 
\[\partial \overline{M_{0,n}} = \bigcup_S D^S,\quad ~\text{ for all }~ S \subseteq \{1,\dots,n\} ~\text{ with }~ 2\leq |S| \leq n-2.\] 
The subscheme $D^S$ describes those stable curves having a node separating all labels in $S$ from all labels in $S^c$, the complement of $S$. Note that $D^S = D^{S^c}$. For example the point $0\in \PP^1_k \cong \overline{M_{0,4}}$ corresponds to the divisor $D^{\{1,4\}} = D^{\{2,3\}}$ (the markings $\{1,2,3\}$ are send to $\{0,1,\infty\}$ and the fourth is used as the coordinate on $\PP^1_k$; the divisor $D^{\{1,4\}}$ describes exactly what happens when the coordinate becomes zero, i.e.\ markings $1$ and $4$ collide). 

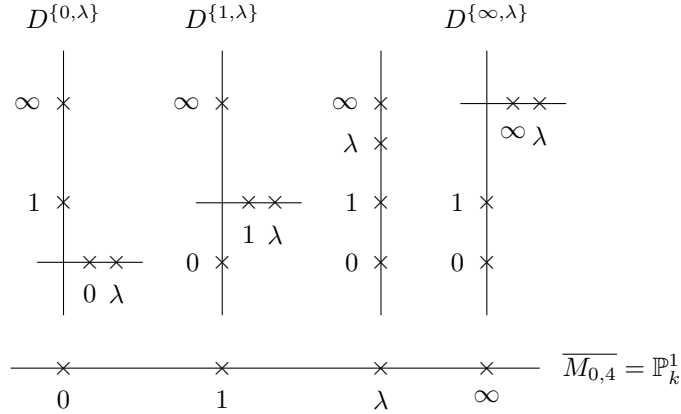
\begin{figure}[!h]
	\centering
	\begin{tikzpicture}[scale = 0.7]
		\draw (0,0) -- (10,0); 
		\node[right = 5pt] at (10,0) {$\overline{M_{0,4}}=\mathbb{P}^1_k$};
		
		\node at (1,0) {$\times$};
		\node[below = 5pt] at (1,0) {$0$};
		\node at (4,0) {$\times$};
		\node[below = 5pt] at (4,0) {$1$};
		\node at (9,0) {$\times$};
		\node[below = 5pt] at (9,0) {$\infty$};
		\node at (7,0) {$\times$};
		\node[below = 5pt] at (7,0) {$\lambda$};
		
		\draw (1,1) -- (1,6);
		\draw (0.5,2) -- (2.5,2);
		\node at (1.5,2) {$\times$};
		\node[below = 5pt] at (1.5,2) {$0$};
		\node at (2,2) {$\times$};
		\node[below = 5pt] at (2,2) {$\lambda$};
		\node at (1,2+3*3/8) {$\times$};
		\node[left = 5pt] at (1,2+3*3/8) {$1$};
		\node at (1,5) {$\times$};
		\node[left = 5pt] at (1,5) {$\infty$};
		\node[above = 5pt] at (1,6) {$D^{\{0,\lambda\}}$};
		
		\draw (4,1) -- (4,6);
		\node at (4,2) {$\times$};
		\node[left = 5pt] at (4,2) {$0$};
		\draw (3.5,2+3*3/8) -- (5.5,2+3*3/8);
		\node at (4.5,2+3*3/8) {$\times$};
		\node[below = 5pt] at (4.5,2+3*3/8) {$1$};
		\node at (5,2+3*3/8) {$\times$};
		\node[below = 5pt] at (5,2+3*3/8) {$\lambda$};
		\node at (4,5) {$\times$};
		\node[left = 5pt] at (4,5) {$\infty$};
		\node[above = 5pt] at (4,6) {$D^{\{1,\lambda\}}$};
		
		\draw (9,1) -- (9,6);
		\node at (9,2) {$\times$};
		\node[left = 5pt] at (9,2) {$0$};
		\node at (9,2+3*3/8) {$\times$};
		\node[left = 5pt] at (9,2+3*3/8) {$1$};
		\draw (8.5,5) -- (10.5,5);
		\node at (9.5,5) {$\times$};
		\node[below = 5pt] at (9.5,5) {$\infty$};
		\node at (10,5) {$\times$};
		\node[below = 5pt] at (10,5) {$\lambda$};
		\node[above = 5pt] at (9,6) {$D^{\{\infty,\lambda\}}$};
		
		\draw (7,1) -- (7,6);
		\node at (7,2) {$\times$};
		\node[left = 5pt] at (7,2) {$0$};
		\node at (7,2+3*3/8) {$\times$};
		\node[left = 5pt] at (7,2+3*3/8) {$1$};
		\node at (7,5) {$\times$};
		\node[left = 5pt] at (7,5) {$\infty$};
		\node at (7,2+3*6/8) {$\times$};
		\node[left = 5pt] at (7,2+3*6/8) {$\lambda$};
	\end{tikzpicture}
\end{figure}

There is an isomorphism $D^S \cong \overline{M_{0,|S|+1}} \times \overline{M_{0,|S^c|+1}}$, taking the two branches and remembering the node as an additional marking. This recursive structure of building $\overline{M_{0,n}}$ out of $M_{0,n}$ and the divisors $D^S$, which are again products of two $\overline{M_{0,n'}}$ for smaller $n'$, will be important for the computation in \cref{sec examples}. The intersection of two divisors $D^S\cap D^T$ is non-empty if and only if $S\subseteq T$, $T\subseteq S$, $S^c \subseteq T$ or $T^c\subseteq S$. The last two cases can be summarized by $S\cap T = \emptyset$. In a similar way as before, we have $D^S\cap D^T \cong \overline{M_{0,a_1}}\times \overline{M_{0,a_2}}\times \overline{M_{0,a_3}}$, for some $a_i< n$, if the intersection is non-empty.

There are morphisms $\varphi_P\colon \overline{M_{0,n}}\to \overline{M_{0,|P|}}$ for $P\subseteq \{1,\dots,n\}$ with $|P|\geq 3$, forgetting all labels except the ones in $P$ and contracting all unstable components. The morphisms $\overline{M_{0,n_1+1}}\times \overline{M_{0,n_2+1}}\to \overline{M_{0,n_1+n_2}}$, gluing two stable curves at a marking, gives the collection of all $\overline{M_{0,n}}$ the structure of a topological operad, see \cite{OperadsBook}. 

\subsection{Classical computations}
Known classical invariants of $\overline{M_{0,n}}$ include the Chow ring and from that also the singular cohomology of the complex points $H^{2\ast}_\text{sing}(\overline{M_{0,n}},\ZZ)\cong \CH^\ast(\overline{M_{0,n}})$, see \cite{Keel}. They are isomorphic as rings but the grading is multiplied by two in the singular cohomology. The Chow ring $\CH^\ast(\overline{M_{0,n}})$ is generated in degree $1$ by the classes $D^S$ described above for $S\subseteq\{1,\dots,n\}$ with $2\leq |S| \leq n-2$ subject to the following three relations:
\begin{compactenum}[(i)]
    \item $D^S = D^{S^c}$
    \item For distinct elements $i,j,k,l\in\{1,\dots,n\}$:
    \[\sum_{\substack{i,j\in S \\ k,l\notin S}} D^S = \sum_{\substack{i,k\in S \\ j,l\notin S}} D^S = \sum_{\substack{i,l\in S \\ j,k\notin S}} D^S\]
    \item $D^S D^T = 0$ unless $S\subseteq T$, $T\subseteq S$ or $S\cap T = \emptyset$.
\end{compactenum}
The first relation is due to the geometric description. The three summands in condition (ii) are exactly the pullbacks of the boundary divisors of $\overline{M_{0,4}}$ under the map $\varphi_{\{i,j,k,l\}}$ forgetting all labels except $i$,$j$,$k$,$l$ and stabilizing. The vanishing in the third condition is equivalent to $D^S$ and $D^T$ not intersecting. 

The computation in \cite[Theorem 1]{Keel} describes $\overline{M_{0,n}}$ as an iterated blow-up along smooth subschemes and uses the blow-up formula for Chow rings (i.e.\ localization sequence and projective bundle formula). 

This is something that is not directly possible in the Milnor-Witt setting, because the projective bundle formula does not hold (the localization sequence still holds). We will see this later by directly computing some cohomology groups of $\PP^n_k$. This phenomenon mirrors the $2$-torsion of the singular cohomology of $\RP^n$.

The additive structure of $H^\ast_\text{sing}(\overline{M_{0,n}}(\RR),\ZZ)$ was computed in \cite{EHKR} by computing the rational cohomology from the operad structure, the $2$-torsion via the Bockstein sequence and showing that there is no other torsion.

\subsection{Milnor-Witt (co)homology}
Let $F$ be a field and denote by $K^{MW}_*(F)$ the $\ZZ$-graded ring generated by symbols $[a]$ of degree $1$ for $a\in F^\times$ and an additional symbol $\eta$ in degree $-1$ satisfying the following relations:
\begin{compactenum}[(i)]
	\item \makebox[11em][l]{$[a][1-a]=0$} $\in K^{MW}_2(F)$ for $a,1-a\in F^\times$
	\item \makebox[11em][l]{$[ab]=[a]+[b]+\eta[a][b]$} $\in K^{MW}_1(F)$ for $a,b\in F^\times$
	\item \makebox[11em][l]{$\eta [a] = [a]\eta$} $\in K^{MW}_0(F)$ for $a\in F^\times$
	\item \makebox[11em][l]{$(2+\eta[-1])\eta = 0$} $\in K^{MW}_{-1}(F)$
\end{compactenum}
Notation: $[a_1,\dots, a_n] = [a_1]\dots [a_n]$. 

By definition $K^{MW}_\ast(F)/(\eta) = K^M_\ast(F)$. The ring $K^{MW}_0(F)$ is isomorphic to the Grothendieck-Witt ring $GW(F)$, i.e.\ non-degenerate symmetric bilinear forms over $F$ up to isometry, via $\langle a \rangle \mapsto 1 + \eta[a]$. Here $\langle a \rangle$ for $a\in F^\times$ denotes the form $F\times F\to F, (x,y)\mapsto axy$. 

Multiplication by $\eta$ gives an isomorphism $K^{MW}_q(F) \to K^{MW}_{q-1}(F)$ for $q\leq -1$ and all these groups are isomorphic to the Witt ring $W(F)=GW(F)/(1 + \langle -1\rangle)$. 

Denote by $I(F) = \ker(\operatorname{rk}\colon GW(F)\to \ZZ)$ the fundamental ideal, by $I^n(F)$ its powers and $\bar{I}^n(F) = I^n(F)/I^{n+1}(F)$. There is a morphism $K^{MW}_n(F) \to I^n(F)$ given by \[\eta^m[u_1,\dots,u_{n+m}]\mapsto \langle -1, u_1\rangle \dots \langle -1,u_{n+m}\rangle\in I^{n+m}(F)\subseteq I^{n}(F).\]
All these objects together give the following pullback square
\[\begin{tikzcd}[column sep=small, row sep=small]
	K^{MW}_n(F) \arrow[rrr]\arrow[ddd] &&& K^{M}_n(F)\arrow[dd]\\ &&&\\ &&&K^{M}_n(F)/2\arrow[dl, end anchor=north east, start anchor = south west, "\cong"]\\ I^n(F) \arrow[rr] && I^n(F)/I^{n+1}(F)&\\
\end{tikzcd}\]
where the isomorphism $K^{M}_n(F)/2 \to I^n(F)/I^{n+1}(F)$ induced by $[u]\mapsto\langle -1,u\rangle$ comes from the resolution of the Milnor conjecture on quadratic forms, see \cite{OVVMilnorConj} for a proof in characteristic $0$ or \cite{DuggerNotesMilnorConjectures} for a survey.

For $X$ a separated, finite type $k$-scheme and $\mathcal{L}$ a line bundle on $X$, the Rost-Schmid complex is defined in degree $i$ as \[C_{RS}(X,K^{MW}_j(\mathcal{L}))_i := \bigoplus_{x\in X_{(i)}} K^{MW}_{j+i}(\kappa(x), \det(\mathrm\Omega_{\kappa(x)/k})\otimes\mathcal{L}_{\kappa(x)}).\]
with the differential defined in terms of residue maps. The sum is indexed by $X_{(i)}$, the points of dimension $i$. The groups appearing in those summands are all of the form \[K^{MW}_q(F,L):= K^{MW}_q(F)\otimes_{\ZZ[F^\times]} \ZZ[L\setminus\{0\}]\] for a one-dimensional $F$ vector space $L$. They are non-canonically isomorphic to $K^{MW}_q(F)$. The reason for this is that, in contrast to Milnor K-theory, the residue maps here depend on choosing uniformizers. Introducing the twist is a way to record the choice of uniformizer \cite{FaselChowWitt}.

Denote the homology of $C_{RS}(X,K^{MW}_j(\mathcal{L}))$ by $H^{RS}_i(X,K^{MW}_j(\mathcal{L}))$. This leads to the following identifications
\begin{align*}
    H^{RS}_i(X,K^{MW}_q(\mathcal{L})) &= H^{BM}_i(X,\underline{K}^{MW}_{\dim(X)+q}(\mathcal{L}))\\ &= H^{\dim(X)-i}(X,\underline{K}^{MW}_{\dim(X)+q}(\mathcal{L}\otimes \omega_{X/k})),
\end{align*}
where the middle term is the Borel-Moore homology of $\underline{K}^{MW}_{\dim(X)+q}(\mathcal{L})$ and the latter identification, i.e.\ Poincaré duality, requires $X$ to be smooth, see \cite[Theorem 4.2.11 in Chapter 6]{MilnorWittMotives}. The only advantage of using the Rost-Schmid notation is the fact that the dimension of spaces does not cause index shifts. For example, the localization sequence for a closed immersion $i\colon Z\hookrightarrow X$ with open complement $j\colon U\hookrightarrow X$ and line bundle $\mathcal{L}$ on $X$ is \[\dotsc \to H^{RS}_l(X, K^{MW}_m(\mathcal{L})) \overset{j^\ast}{\longrightarrow} H^{RS}_l(U, K^{MW}_m(\mathcal{L})) \rightarrow  H^{RS}_{l-1}(Z, K^{MW}_m(\mathcal{L}))\to \dotsc.\] 
In both other versions, the sheaf index and the cohomological index of the groups for $Z$, have an additional \enquote{$-\operatorname{codim}_X(Z)$}.

By definition $\CH_i(X) = H^{RS}_i(X,K^{M}_{-i})$, where the $K^{M}_{-i}$ is purely notational, as $K^{M}_{i}(F) = 0$ for $i<0$. For $X$ smooth, we have $\CH^i(X) = H^i(X,\underline{K}^{M}_{i})$, also known as Bloch's formula. For $k\hookrightarrow\RR$, the real cycle class map is an isomorphism $H^i(X,\underline{I}^q(\mathcal{L})) \to H^i_\text{sing}(X(\RR),\ZZ(\mathcal{L}))$ for $q\geq \dim(X)+1$, see \cite{RealCycleClassMap}.

\section{Morel--Sawant cellularity}
\subsection{Classically}
Classically, a cellular structure (which we call a \enquote{strict cellular structure}) on a scheme $X$ is an increasing filtration
\[\emptyset = \Omega_{-1}\subseteq \Omega_0 \subseteq \dotsc \subseteq \Omega_d=X,\]
with $\Omega_i$ open and $(\Omega_i\setminus \Omega_{i-1})_\text{red}$ is the disjoint union of finitely many $\AA^{n-i}_k$.

From this description we can immediately read off the Chow groups \[\CH_{n-i}(X) = \ZZ[\text{closure of the $\AA^{n-i}_k$ above}].\]
Sometimes it is more convenient to filter by the closed complements $X\setminus \Omega_i$. 

For example we can filter $\PP^n_k$ by closed strata as 
\[\PP^n_k\supseteq \PP^{n-1}_k \supseteq \PP^{n-2}_k \supseteq \dotsc \supseteq \PP^1_k \supseteq \PP^0_k=\{\ast\}. \]
Each consecutive difference is a single $\AA^i_k$ leading to 
\[\CH_i(\PP^n_k) \cong \begin{cases}\ZZ,\quad &i\in\{0,\dots,n\} \\ 0, \quad&\text{otherwise}\end{cases}.\]

We would like to stratify $\overline{M_{0,n}}$ by closed subschemes depending on the number of divisors they lie in, i.e.\ the number of nodes they have. The big open cell is then $M_{0,n}$. The codimension $1$ cells all look like 
\[D^S \setminus \left(\bigcup_{T\neq S}D^S\cap D^T\right)\cong M_{0,|S|+1}\times M_{0,|S^c|+1}.\]
All of these cells are complements of hyperplane arrangements, i.e.\ finite unions of hyperplanes, in some $\AA^d_k$. These are not too complicated, but do not fit the notion of a strict cellular structure.
\subsection{Improved: Cohomologically trivial cells}
The notion of a cellular structure from Morel-Sawant in \cite{MorelSawant}, allows more general cells, while keeping many arguments intact.
\begin{defi}
    A cellular structure on a scheme $X$ is an increasing filtration
\[\emptyset = \Omega_{-1}\subseteq \Omega_0 \subseteq \dotsc \subseteq \Omega_d=X,\]
with $\Omega_i$ open and $(\Omega_i\setminus \Omega_{i-1})_\text{red}$ is $k$-smooth, affine, everywhere of codimension $i$ and cohomologically trivial (i.e.\ $H_{Nis}^j(\Omega_i\setminus \Omega_{i-1},\underline{M}) = 0$ for all $j > 0$ and all strictly $\AA^1$-invariant sheaves of abelian groups $\underline{M}$).
\end{defi}

In the following definition of  cellular complexes we decided to use homological Rost-Schmid notation, instead of cohomological notation. The reason for this is only notational; in the homological version there are no index shifts for pullbacks or contractions of sheaves involved in the localization sequences.

\begin{defi}
	Let $\emptyset = \Omega_{-1} \subseteq \Omega_0 \subseteq \Omega_1 \subseteq \dots \subseteq \Omega_d = X$ be a cellular structure and $\underline{M}$ a strictly $\AA^1$-invariant sheaf. Define the cellular complex 
	\begin{align*}
		C^\mathrm{cell}_i(X,\underline{M})= H_{\dim(X)-i}(\Omega_i\setminus\Omega_{i-1},\underline{M})
	\end{align*}
	with differential $d_{\dim(X)-i}$ given by the composition
	\[H_{\dim(X)-i}(\Omega_{i}\setminus\Omega_{i-1},\underline{M})\overset{\iota_*}{\longrightarrow} H_{\dim(X)-i}(\Omega_{i},\underline{M}) \overset{\partial}{\longrightarrow} H_{\dim(X)-i-1}(\Omega_{i+1}\setminus\Omega_{i},\underline{M}),\]
	coming from the two appropriate localization sequences.
\end{defi}

\begin{thm}
    Let $\emptyset = \Omega_{-1} \subseteq \Omega_0 \subseteq \dots \subseteq \Omega_d = X$ be a cellular structure and $\underline{M}$ a strictly $\AA^1$-invariant sheaf. The cellular complex $C^\text{cell}_*(X,\underline{M})$ is a complex and for all $k$:
	\[H_k(C^\text{cell}_*(X,\underline{M})) \cong H_k(X,\underline{M}).\]
\end{thm}
The proof of the theorem is completely dual to the proof in classical topology. The cohomology of a smooth scheme $X$ is obtained by the following translation \[H^i(X,\underline{M}) = H_{\dim(X)-i}(X,\underline{M}_{-\dim(X)}(\omega_X)),\] for the $\dim(X)$-contraction of $\underline{M}$ twisted by the canonical bundle $\omega_X$ of $X$.

\section{How to do computations and some examples}
\subsection{General technique}
The computations all follow the same pattern:
\begin{compactenum}[(i)]
    \item Find a cellular structure (if there exists one)
    \item Compute all $H^0(\text{cells}, \underline{M})$
    \item Find curves in those cells meeting the boundary nicely
    \item Compute the restriction for enough curves to know the complete differential
\end{compactenum}
The last two steps make computing the differential easier. The only curves we restrict to are ones that are a regularly embedded $\PP^1_k$. The advantage is that we know how to compute twists, coming from the normal bundle of the curve in that cell, and differentials on $\PP^1_k$, see next section. The number of curves one has to restrict to depends on the $H^0(\text{cells}, \underline{M})$ involved. Given a strictly cellular structure, i.e.\ cells are $\AA^m_k$, we have $H^0(\AA^m_k, \underline{M})\cong \underline{M}(k)$. Therefore a single curve suffices to determine the residue of a form on that cell. 

More precisely, we can understand the differential in the localization sequence evaluated at a point by first restricting to a curve and computing the differential there, by the following lemma.

\begin{lem}\label{lem: restriction to curves}
	Let $Z\hookrightarrow X$ be a regular embedding of codimension $1$. Let $C\hookrightarrow X$ be a regularly embedded curve meeting $Z$ transversally in point $p$. Then the following diagram $\langle -1\rangle$-commutes 
	\[\begin{tikzcd} 
		H_{\dim(X)}(X\setminus Z,K^{MW}_j(\mathcal{L})) \arrow[r, "\partial"]\arrow[d,"\iota^*"'] & H_{\dim(X)-1}(Z,K^{MW}_j(\mathcal{L}))\arrow[d, "\iota^*"] \\ 
		H_{1}(C\setminus p,K^{MW}_{j+d-1}(\det(\mathcal{N}_{C/X})^{-1}\otimes\mathcal{L}|_C)) \arrow[r,"\partial"'] & H_{0}(p,K^{MW}_{j+d-1}(\mathcal{L})),
	\end{tikzcd}\]
	i.e.\ $\partial\circ \iota^\ast = \langle -1\rangle \iota^\ast\circ\partial$, where the horizontal morphisms come from the appropriate localization sequences.
\end{lem}

The fact that the diagram above is $\langle -1\rangle$-commutative instead of honestly commuting is irrelevant as long as all our computations are done via the lemma. 

We will only restrict to regularly embedded $\PP^1_k$, because there we know how to compute the differentials, see below, and  $\det(\mathcal{N}_{C/X})=(\omega_{X/k})_{|_C}\in \Pic(\PP^1_k)/2$ by adjunction formula.

\subsection{Examples}\label{sec examples}
For all examples we fix a perfect field $k$ with $\operatorname{char}(k)\neq 2$. 

To simplify the notation we are considering $\underline{M} = \underline{K}^{MW}_q(\mathcal{L})$ for some $q\in \ZZ$. Other sheaves like $\underline{I}^q(\mathcal{L})$, $\underline{\bar{I}}^q$ and $\underline{K}^M_q$ follow immediately from this computation by taking the appropriate quotient maps on chain complexes. Note that twists by line bundles $\mathcal{L}$ do not matter for the latter two. 

The upcoming computation has an important notational disadvantage. One is interested in the Borel-Moore homology or cohomology of the sheaf $\underline{M}$. The Rost-Schmid complex to compute it uses $M_{-\dim(X)}$ coefficients. This might lead to the impression that we assume our coefficients are contractions themselves. This is not the case. It is only notational misfortune.

\subsubsection{Projective line}
All computations done here rely on a complete understanding of the $\PP^1_k$ case. Fix a line bundle $\mathcal{L}$ on $\PP^1_k$. Consider the cellular structure on $\PP^1_k$ given by picking $n$ distinct $k$-rational points \[\PP^1_k \supseteq \{\infty, p_2,\dots, p_n\}.\]
In general, it is not necessary to only consider $k$-rational points. The exact same argument works\footnote{This argument requires at least one point to be $k$-rational, e.g.\ $\infty$. If no point is $k$-rational the open cell $U=\PP^1_k\setminus\{p_1,\dots,p_n\}$ might not be cohomologically trivial, e.g.\ $\CH^1(U) = \ZZ/\gcd(\deg(p_1),\dots,\deg(p_n))\ZZ$.} if one uses the residue fields of the points instead of $k$, i.e.\ $K_q^{MW}(\kappa(p_i))$ instead of $K_q^{MW}(k)$.
The corresponding cellular complex is
\[K^{MW}_{q+1}(k) \oplus (K^{MW}_{q}(k))^{\oplus n-1} \overset{d_1}{\longrightarrow} K^{MW}_q(k) \oplus (K^{MW}_{q}(k))^{\oplus n-1}.\]
The $K^{MW}_{q+1}(k)$ summand describes constant forms on $\PP^1_k\setminus\{\infty,p_2,\dots, p_n\}$. The other summands on the left-hand side are forms having a simple pole at one $p_i$, and correspond to \[[t-p_i]K^{MW}_q(k) \subseteq H_1\left(\PP^1_k\setminus\{\infty,p_2,\dots, p_n\},K^{MW}_q(\mathcal{L})\right) \subseteq K^{MW}_{q+1}(k(t)).\]
The right-hand side summands are the appropriate residues at $\{\infty, p_2,\dots, p_n\}$.

The differential depends on the class of $\mathcal{L}\cong \mathcal{O}_{\PP^1_k}(\ell) \in \Pic(\PP^1_k)/2 = \ZZ/2\ZZ$, hence on the parity of $\ell$. Writing the differential $d_1$ in matrix form corresponding to the direct sum decomposition above results in 
\begin{align*}
	\begin{pmatrix}\eta & -1 & \dots & -1 \\ 0 & 1 & & \\ \vdots & & \ddots &\\ 0 & & & 1 \end{pmatrix},\quad \ell\text{ odd}\qquad\qquad \begin{pmatrix}0 & \varepsilon & \dots & \varepsilon \\ 0 & 1 & & \\ \vdots & & \ddots &\\ 0 & & & 1 \end{pmatrix},\quad \ell\text{ even}
\end{align*}
where all entries outside of the first row and main diagonal are zero and we use $\varepsilon = -\langle -1\rangle \in K^{MW}_0(F)$. 
This computes $H^i(\PP_k^1, \underline{K}^{MW}_q(\mathcal{L})) = H_{1-i}(\PP_k^1, \underline{K}^{MW}_{q-1}(\mathcal{L}))$, because the canonical bundle $\omega_{\PP^1_k} = \mathcal{O}_{\PP^1_k}(-2)$ is a square:
\begin{align*}
	H^0\left(\PP_k^1, \underline{K}^{MW}_q(\mathcal{O}(\ell))\right) &= \begin{cases} _\eta K^{MW}_q(k), & \ell \text{ odd}, \\ K^{MW}_q(k){,}\phantom{\text{$/\eta$}} & \ell \text{ even}, \end{cases}\\ 
	H^1\left(\PP_k^1, \underline{K}^{MW}_q(\mathcal{O}(\ell))\right) &= \begin{cases} K^{MW}_{q-1}(k)/\eta, & \ell \text{ odd}, \\ K^{MW}_{q-1}(k), & \ell \text{ even}. \end{cases}
\end{align*}
where $_\eta K^{MW}_{q}(k) = \ker (K^{MW}_{q}(k) \overset{\cdot\eta}{\longrightarrow} K^{MW}_{q-1}(k))$.

From this we can (almost) immediately read off 
\begin{align*}
	H^{2q}_\text{sing}(\CP^1, \ZZ) \cong \CH^q(\PP_k^1) \cong H^q(\PP_k^1,\underline{K}^M_q) &= \begin{cases} \ZZ,\phantom{\text{$/2\ZZ$}} & q=0,1, \\ 0, &\text{otherwise}.\end{cases}\\ 
    H^q_\text{sing}(\RP^1, \ZZ) \cong H^q(\PP_\RR^1,\underline{I}^q) &= \begin{cases} \ZZ,\phantom{\text{$/2\ZZ$}} & q=0,1, \\ 0, & \text{otherwise},\end{cases}\\
	H^q_\text{sing}(\RP^1,\ZZ(\mathcal{O}(-1))) \cong H^q(\PP_\RR^1,\underline{I}^q(\mathcal{O}(-1))) &= \begin{cases} \ZZ/2\ZZ, & q=1, \\ 0, & \text{otherwise}.\end{cases} 
\end{align*}
The isomorphism from singular cohomology of the real points (with twisted coefficients) to $\underline{I}^q(\mathcal{L})$-coefficients is given by the real cycle class map, see \cite{RealCycleClassMap}. Here it is enough to consider the $q$-th fundamental ideal power, because $\PP^1_k$ is strictly cellular, see \cite[Theorem 5.7]{RealCycleClassMap} or \cite[Corollary 4.13]{RealCycleIsomLinear}.

\subsubsection{Projective spaces and products}
Consider the standard strict cellular structure on $\PP_k^n$ given by fixing a complete flag of subspaces $\emptyset\subseteq\PP_k^0\subseteq \PP_k^1 \subseteq\dots\subseteq\PP_k^n$. Each cell is a single $\AA^i_k$, so the chain complex for coefficients in $K^{MW}_q(\mathcal{L})$ is given by
\[\begin{tikzcd}
	K^{MW}_{q+n}(k) \arrow[r, "d_n"] & K^{MW}_{q+n-1}(k) \arrow[r, "d_{n-1}"] & \dots \arrow[r, "d_{2}"] & K^{MW}_{q+1}(k) \arrow[r, "d_{1}"] & K^{MW}_{q}(k).
\end{tikzcd}\]
To compute the differential $d_i\colon K^{MW}_{q+i}(k)\to K^{MW}_{q+i-1}(k)$ we consider the cell of dimension $i$. Assume $\PP^i_k$ is given by ${[x_0:\dotsc:x_i:0:\dotsc:0]}$. The rational curve given by ${[0:\dotsc:0:x_{i-1}:x_i:0:\dotsc:0]}$ lies in $\PP^{i}_k$ and meets the boundary $\PP^{i-1}_k$ transversally. The determinant of the normal bundle of this curve is $\mathcal{O}_{\PP_k^1}(-i+1)$. Let $\mathcal{L}=\mathcal{O}_{\PP_k^n}(j)$ be the twist and observe that restricting to cells keeps the parity, i.e.\ $\mathcal{O}_{\PP_k^n}(j)|_{\PP_k^{n-1}}=\mathcal{O}_{\PP_k^{n-1}}(j)$. Therefore, the twist bundle appearing in the restriction lemma for differential $d_i$ is given $\mathcal{O}_{\PP^1_k}(j-i+1)$. Hence, the $\PP^1_k$ computation shows that $d_i$ is given by multiplication with:
\begin{align*}
	\begin{cases}\eta, \text{ for }j-i+1\text{ odd},\\ 0, \text{ for }j-i+1\text{ even}.\end{cases}
\end{align*}
After all necessary index shifts and twist changes by $\omega_{\PP^n_k}=\mathcal{O}_{\PP^n_k}(-n-1)$, this results in 
\begin{align*}
	H^i\left(\PP_k^n, \underline{K}^{MW}_q\right) &= \begin{cases} K^{MW}_q(k), & i=0,\\ K^{MW}_{q-i}(k)/\eta, & 2\leq i\leq n \text{ and $i$ even},\\ _\eta K^{MW}_{q-i}(k), & 1\leq i<n \text{ and $i$ odd},\\ K^{MW}_{q-n}(k), & i=n \text{ odd},\end{cases}\\ 
	H^i\left(\PP_k^n, \underline{K}^{MW}_q(\mathcal{O}(-1))\right) &= \begin{cases} K^{MW}_{q-i}(k)/\eta, & 1\leq i \leq n\text{ and $i$ odd}\\ _\eta K^{MW}_{q-i}(k), & 0\leq i < n \text{ and $i$ even},\\ K^{MW}_{q-n}(k), & i= n \text{ even}.\end{cases}
\end{align*}

A quick check can be done by considering $q=i$ and using the quotient map to $K^M_\ast$, giving the known $\CH^i(\PP^n_k) = \ZZ = K^M_0(k)$ for $i=0,\dots,n$. For $k = \RR$, using the quotient to $I^\ast$ gives the groups $\ZZ$, $\ZZ/2\ZZ$, and $_2\ZZ = 0$ respectively. This gives the singular cohomology of $\RP^n$. 

Here we can also see that the projective bundle formula for Chow-Witt groups $\widetilde{\CH}^i(X,\mathcal{L}) = H^i\left(X, \underline{K}^{MW}_i(\mathcal{L})\right)$ does not hold, as $H^i\left(\PP^n_k, \underline{K}^{MW}_i(\mathcal{L})\right)$ it is not free over $H^0(\Spec(k),\underline{K}^{MW}_0(\mathcal{L})) = K^{MW}_0(k)$.

The cohomology of $\PP^n_k$ with values in $\underline{K}_q^{MW}(\mathcal{L})$ was already computed in \cite{FaselProjBund}. The computation there uses a replacement of the projective bundle formula. The cellular computation in \cite{MorelSawant} for coefficients in $\underline{K}_q^{MW}$ uses the $\mathbb{G}_m$-torsor $\AA^{n+1}_k\setminus \{0\} \to \PP^n_k$ and requires a lot of work to compute the differential.

To compute the cohomology groups of $\PP^n_k \times \PP^m_k$, use the product of strictly cell\-ular structures for both factors. Writing down rational curves works exactly the same as in the $\PP^n_k$ case. The canonical bundle is \mbox{$\omega_{\PP^n_k \times \PP^m_k} = \mathcal{O}(-n-1,-m-1)$}. The computation is identical to the cellular one for $H_\text{sing}^*(\RP_k^n\times \RP_k^m, \ZZ)$, replacing the multiplication by $\pm 2$ with the multiplication by $\langle \pm 1\rangle\eta$. Twisting by $\mathcal{O}(d,e)$ changes whether these multiplications appear in even or odd degrees.

Another construction one can do is the following. Let $Y$ be a cohomologically trivial scheme and $\underline{M}_q$ be a homotopy module, see \cite{MorelBook}. A cellular computation shows that we have an isomorphism of cohomology groups
\[H^c(\PP^c_k\times Y,\underline{M}_{q+c}(\mathcal{O}(c+1))) \cong H^0(Y,\underline{M}_q).\]
A cellular structure on $\PP^c_k\times Y$ is induced from the one on $\PP^c_k$ and has as cells of the form $\AA^i_k \times Y$ for $i=0,\dots,c$. The cellular complex around degree $c$ looks like
\[\begin{tikzcd}
    H^0(\AA^1_k\times Y, \underline{M}_{q+1}) \arrow[r, "d^{c-1}"] & H^0(Y, \underline{M}_{q})\arrow[r] & 0.
\end{tikzcd}\]
Similar to the $\PP^c_k$ case  before, the differential $d^{c-1}$ vanishes if we consider twists by $\mathcal{O}(0)$ if $c$ is odd and by $\mathcal{O}(-1)$ if $n$ is even. The twist by $\mathcal{O}(c+1)$ takes care of this. This is the top degree part of a Künneth isomorphism. It was used in \cite{RealCycleIsomLinear} to show that $\PP^c_k \times (\mathbb{G}_m)^{d-c}$ has the conjectured upper bound for the exponent of the real cycle class map's $H^c(X,\underline{I}^c(\mathcal{L}))\to H^c_\text{sing}(X(\mathbb{R}),\mathbb{Z}(\mathcal{L}))$ cokernel.

\subsubsection{The moduli space $\overline{M_{0,n}}$}
Fix a perfect infinite field with $\operatorname{char}(k) \neq 2$. The additional assumption on $k$ is due to a geometric argument later on, which requires $k$ to have enough elements. The big downside of this computation is the amount of non-canonical choices that need to be fixed. Writing down such choices and understanding the differences between them is easy. It is the amount of choices and bookkeeping that makes it tedious.  

As described earlier the first step is to understand the cells. Recall that we stratify $\overline{M_{0,n}}$ by the number of divisors a point lies in. All cells are disjoint unions of products of $M_{0,a}$ for various $a \leq n$. In particular, they are (disjoint unions of) complements of affine hyperplanes in some affine space.
\begin{lem}
	Let $U = \AA_k^n\setminus\left(\bigcup_{i=1}^r V_i\right)$ be the complement of affine hyperplanes $V_i$, then $U$ is cohomologically trivial and 
	\begin{align*}
		H^0(U,\underline{K}^{MW}_q) \cong \bigoplus_{i=0}^n \left(K^{MW}_{q-i}(k)\right)^{m_i},
	\end{align*}
	where $m_0=1$ and $m_i = \big|\big\{J\subseteq\{1,\dots,r\}~\big|~ |J|=i,~ \bigcap_{j\in J}V_j\neq\emptyset\big\}\big|$ is the number of sets of $i$ hyperplanes having non-empty intersection.

    %The isomorphism \[\bigoplus_{i=0}^n \left(K^{MW}_{q-i}(k)\right)^{m_i} \to H^0(U,K^{MW}_q) \subseteq K^{MW}_q(k(t_1,\dots,t_r))\], is given by picking defining equations $\{f_i=0\}=V_i$ and multiplying a summand $K^{MW}_{q-|J|}(k)$ corresponding to $J=\{j_1<\dots< j_{|J|}\}\subseteq\{1,\dots,r\}$ by $[f_{j_{|J|}}]\dots [f_{j_1}]$.
\end{lem}

\begin{proof}
    We start by showing cohomological triviality, i.e.\ $H^q(U,\underline{M}) = 0$ for all $q\geq 1$ and strictly $\AA^1$-invariant sheaves $\underline{M}$. Proceed by induction on the number of hyperplanes $r$.
	\begin{compactenum}
		\item[] $r=0$: Clear, from $U = \AA^n_k$ and the strict $\AA^1$-invariance of $\underline{K}^{MW}_q$ and $\underline{M}$.
		\item[] $r\leadsto r+1$: Consider the localization pair $V_{r+1} \setminus \left(\bigcup_{i=1}^r V_i\right) \hookrightarrow \AA_k^n\setminus \left(\bigcup_{i=1}^r V_i\right)$ with complement $U= \AA_k^n\setminus \left(\bigcup_{i=1}^{r+1} V_i\right)$. Writing \[V_{r+1} \setminus \left(\bigcup_{i=1}^r V_i\right) = V_{r+1} \setminus \left(\bigcup_{i=1}^r V_i\cap V_{r+1}\right)\] and using that $V_i\cap V_{r+1}$ is either empty or isomorphic to $\AA_k^{n-2}$, shows that this is a complement of at most $r$ hyperplanes in $V_{r+1}\cong\AA_k^{n-1}$ and therefore, cohomologically trivial by induction hypotheses. The localization sequence reads
		$$
		\begin{tikzcd}[column sep=small]
		H^q\left(\AA_k^n\setminus \left(\bigcup_{i=1}^r V_i\right),\underline{M}\right) \arrow[r] & H^q\left(U,\underline{M}\right) \arrow[r,"\partial"] & H^q\left(V_{r+1}\setminus \left(\bigcup_{i=1}^r V_i\right),\underline{M}_{-1}\right),
		\end{tikzcd}
		$$
		where both ends vanish by induction hypothesis, showing the cohomological triviality of $U = \AA_k^n\setminus\left(\bigcup_{i=1}^{r+1} V_i\right)$.

        By induction hypothesis we also have,
		\begin{align*}
			H^0\left(V_{r+1}\setminus \Big(\bigcup_{i=1}^r V_i\Big),\underline{K}^{MW}_{q-1}\right) \cong \bigoplus_{i=0}^n \left(K^{MW}_{q-1-i}(k)\right)^{m_i},
		\end{align*} 
		for $m_0=1$ and $m_i = \left|\big\{J\subseteq\{1,\dots,r\}~\big|~ |J|=i,~ V_{r+1}\cap(\bigcap_{j\in J}V_j)\neq\emptyset\big\}\right|$. By cohomological triviality the localization sequence in degree $0$ is a split short exact sequence, showing
        \[H^0\left(U,\underline{K}^{MW}_q\right) \cong H^0\left(\AA_k^n\setminus (\bigcup_{i=1}^r V_i),\underline{K}^{MW}_q\right) \oplus H^0\left(V_{r+1}\setminus (\bigcup_{i=1}^r V_i),\underline{K}^{MW}_{q-1}\right).\]
		Adding the corresponding exponents $m_i + m_{i-1}$ for $i=1,\dots,r+1$ gives exactly the stated exponents.
	\end{compactenum}
    \hfill\qedsymbol
\end{proof}

The isomorphism, coming from the splitting of the localization sequence, \[\bigoplus_{i=0}^n \left(K^{MW}_{q-i}(k)\right)^{m_i} \to H^0(U,\underline{K}^{MW}_q) \subseteq K^{MW}_q(k(t_1,\dots,t_r)),\] is given by picking defining equations $\{f_i=0\}=V_i$ and multiplying a summand $K^{MW}_{q-|J|}(k)$ corresponding to $J=\{j_1<\dots< j_{|J|}\}\subseteq\{1,\dots,r\}$ by $[f_{j_{|J|}}]\dots [f_{j_1}]$. Therefore the choice of such an isomorphism depends on a choice of order on the set of hyperplanes and choices of defining equations.

For the next step, we want to find rational curves in $\overline{M_{0,n}}$ meeting the boundary in appropriate ways. Note that we mean one-dimensional subschemes $C\hookrightarrow \overline{M_{0,n}}$ here, not points of $\overline{M_{0,n}}$. One possible way to do this is to explicitly write down rational curves in $(\PP^1_k)^{n-3}$ and then take their strict transforms in $\overline{M_{0,n}}$. Up to permutation of coordinates, the curves look like
\[{[t:s]} \mapsto \Bigl({[t + \beta_1 s: s]},\dots, {[t+\beta_r s:s]},{[\gamma_{r+1}:1]},\dots,{[\gamma_{n-3}:1]} \Bigr)\in (\PP^1_k)^{n-3},\]
and $k$ having enough elements guarantees that the parameters can be chosen so that the strict transform meets the boundary as desired. Here are again choices involved as the blowup description of $\overline{M_{0,n}}$ requires a choice of order of markings. Nonetheless this choice is not relevant for the computation. We don't need the explicit description of these curves, only the normal bundle in $\overline{M_{0,n}}$ and the divisors they meet.

To deal with twists by line bundles $\mathcal{L}$ on $\overline{M_{0,n}}$ it is also necessary to pick a basis of $\Pic(\overline{M_{0,n}}) = \ZZ^{2^{n-1}-\binom{n}{2}-1}$ and a trivialization on each cell.

The idea for the complete computation is that all arguments are done within a single $\overline{M_{0,n}}$ factor. Therefore it is enough compute the two top degree differentials. The lower degrees are determined by the top two differentials for $\overline{M_{0,n'}}$ for $n' < n$. To compute the differentials, it is useful to note that the curves we restrict to, come from curves in $(\PP^1_k)^{n-3}$, which can be used to compute the differential on $(\PP^1_k)^{n-3}$ corresponding to the non-strictly cellular structure with $M_{0,n}$ as the big open cell. This lifts the $(\PP^1_k)^{n-3}$ computation up to $\pm \langle \pm1\rangle$ signs determined by the twists.

Performing the computation for $\overline{M_{0,5}}$ leads to
\begin{align*}
	H^0(\overline{M_{0,5}}, \underline{K}^{MW}_q) &= K^{MW}_q(k),\\
	H^1(\overline{M_{0,5}}, \underline{K}^{MW}_q) &= {_\eta K^{MW}_{q-1}(k)} \oplus \left(K^{MW}_{q-1}(k)\right)^4,\\
	H^2(\overline{M_{0,5}}, \underline{K}^{MW}_q) &= K^{MW}_{q-2}(k)/\eta.
\end{align*}
As before, we can read off the following groups
\begin{align*}
	\CH^0(\overline{M_{0,5}})  = H_\text{sing}^{0}(\overline{M_{0,5}}(\CC), \ZZ) &= \ZZ\\
    \CH^1(\overline{M_{0,5}})  = H_\text{sing}^{2}(\overline{M_{0,5}}(\CC), \ZZ) &= \ZZ^5\\
    \CH^2(\overline{M_{0,5}})  = H_\text{sing}^{4}(\overline{M_{0,5}}(\CC), \ZZ) &= \ZZ\\
	H_\text{sing}^0(\overline{M_{0,5}}(\RR), \ZZ) &= \ZZ,\\
	H_\text{sing}^1(\overline{M_{0,5}}(\RR), \ZZ) &= \ZZ^4,\\
    H_\text{sing}^2(\overline{M_{0,5}}(\RR), \ZZ) &= \ZZ/2\ZZ.
\end{align*}
Computations for other twists and higher dimensional $\overline{M_{0,n}}$ can be made as well. Unfortunately, the size of the matrices for the differentials and the number of twists grows quickly. For $\overline{M_{0,5}}$ it is a $(30\times 12)$-matrix for $d_2$ and a $(15\times 30)$-matrix for $d_1$. There are $2^{5-1}-\binom{5}{2}-1 = 5$ possible twists for $\overline{M_{0,5}}$. The $d_3$ differential for $\overline{M_{0,6}}$ is a $(315 \times 105)$-matrix and there are $16$ twists. The number of boundary divisors, the number of summands in the cohomology groups of cells and the number of twists all grow like $O(2^n)$. On the other hand there is nothing more to do than writing down these matrices and computing the (co)homology of the complex.

In general, the whole chain complex for $\overline{M_{0,n}}$ is understandable. More details can be found in the author's PhD thesis \cite{PhDThesis}. A corresponding paper is in preparation.

\begin{conj}
    Let $k$ be any field of characteristic not $2$, $n\geq 3$ be a natural number and $\mathcal{L}$ a line bundle on $\overline{M_{0,n}}$ then
    \[H^i(\overline{M_{0,n}}, \underline{K}^{MW}_q(\mathcal{L}))\cong K^{MW}_{q-i}(k)^{\alpha_{n,i}(\mathcal{L})}\oplus {_\eta K^{MW}_{q-i}(k)}^{\beta_{n,i}(\mathcal{L})} \oplus (K^{MW}_{q-i}(k)/\eta)^{\gamma_{n,i}(\mathcal{L})},\]
    with 
    \begin{compactenum}
        \item[] $\alpha_{n,i}(\mathcal{L}) + \beta_{n,i}(\mathcal{L}) + \gamma_{n,i}(\mathcal{L}) = \operatorname{rk}(\CH^i(\overline{M_{0,n}})) = \operatorname{rk}\left(H_\text{sing}^i(\overline{M_{0,n}}(\CC), \ZZ)\right)$\medskip
        \item[] $\alpha_{n,i}(\mathcal{L}) = \operatorname{rk}\left(H_\text{sing}^i(\overline{M_{0,n}}(\RR), \ZZ(\mathcal{L}))\otimes \mathbb{Q}\right)$ \qquad \enquote{rank of the free part} \medskip
        \item[] $\gamma_{n,i}(\mathcal{\mathcal{L}}) = \operatorname{rk}\left({_2 H_\text{sing}^i(\overline{M_{0,n}}(\RR), \ZZ(\mathcal{L}))}\right)$ \qquad \enquote{rank of the $2$-torsion}
    \end{compactenum}
\end{conj}

\begin{rmk}
    This conjecture should be seen more as a starting point than a goal. More interesting than the structure described above is a description of $\alpha_{n,i}(\mathcal{L})$, $\beta_{n,i}(\mathcal{L})$ and $\gamma_{n,i}(\mathcal{L})$,  geometric descriptions of generators, the ring structure of $\widetilde{\CH}^\ast(\overline{M_{0,n}},\mathcal{L}) = \bigoplus_i H^i(\overline{M_{0,n}}, \underline{K}^{MW}_i(\mathcal{L}))$, and the operad structure.

    The conditions on $\alpha_{n,i}(\mathcal{L})$, $\beta_{n,i}(\mathcal{L})$, $\gamma_{n,i}(\mathcal{L})$ written there are the first easy consequences of complex and real realizations. More are easily found by applying standard arguments like universal coefficient theorems.
\end{rmk}

\section*{Acknowledgments}
    I would like to thank Matthias Wendt for tremendous support and help throught the PhD. Additionally, I would like to thank Leonie Kayser and Matthias Wendt for feedback on a previous draft of these notes.

\end{document}